\documentclass[reqno]{amsproc}
\usepackage{amssymb}
\usepackage{graphicx}
\usepackage{euscript}
\usepackage{mathrsfs} 
\usepackage{units}   
\usepackage{color}

\makeatletter
\@namedef{subjclassname@2010}{%
\textup{2010} Mathematics Subject Classification}
\makeatother

\numberwithin{equation}{section}

\newtheorem{thm}{Theorem}[section]
\newtheorem{cor}[thm]{Corollary}
\newtheorem{lem}[thm]{Lemma}
\newtheorem{pro}[thm]{Proposition}

\theoremstyle{remark}
\newtheorem{rem}[thm]{Remark}

\theoremstyle{definition}
\newtheorem{exa}[thm]{Example}

\DeclareMathOperator{\lin}{\mbox{{\sc Lin}}}
\DeclareMathOperator{\dess}{{\mathsf{Des}}}
\DeclareMathOperator{\dzii}{{\mathsf{Chi}}}

\DeclareMathOperator{\koo}{{\mathsf{root}}}

\DeclareMathOperator{\paa}{{\mathsf{par}}}

\newcommand*{\borel}[1]{{\mathfrak B}(#1)}
\newcommand*{\cbb}{\mathbb C}
\newcommand*{\cfr}{\mathfrak c}

\newcommand*{\D}{\mathrm d}
\newcommand*{\des}[1]{{\dess(#1)}}
\newcommand*{\dz}[1]{{\EuScript D}(#1)}
\newcommand*{\dzi}[1]{\dzii(#1)}
\newcommand*{\dzie}[1]{\dzii_{\mathrm{\,e}}(#1)}
\newcommand*{\dzip}[1]{\dzii^{\prime}(#1)}
\newcommand*{\dzin}[2]{\dzii^{\langle#1\rangle}(#2)}

\newcommand*{\escr}{{\mathscr{E}_V}}
\newcommand*{\ff}{\mathcal F}
\newcommand*{\Ge}{\geqslant}
\newcommand*{\hh}{\mathcal H}
\newcommand*{\hscr}{{\mathscr H}}
\newcommand*{\is}[2]{\langle#1,#2\rangle}
\newcommand*{\jd}[1]{\EuScript N(#1)}
\newcommand*{\kk}{\mathcal K}

\newcommand*{\lambdab}{{\boldsymbol\lambda}}
\newcommand*{\Le}{\leqslant}
\newcommand*{\nbb}{\mathbb N}
\newcommand*{\ogr}[1]{\boldsymbol B(#1)}
\newcommand*{\ob}[1]{{\mathcal R}(#1)}
\newcommand*{\pa}[1]{\paa(#1)}
\newcommand*{\rbb}{\mathbb R}
\newcommand*{\slam}{S_{\boldsymbol \lambda}}

\newcommand*{\tcal}{{\mathscr T}}
\newcommand*{\zbb}{\mathbb Z}

\begin{document}
   \title[Operators with absolute continuity properties]
{Operators with absolute continuity properties:\ \\ an
application to quasinormality}
   \author[Z.\ J.\ Jab{\l}o\'nski]{Zenon Jan Jab{\l}o\'nski}
\address{Instytut Matematyki, Uniwersytet Jagiello\'nski,
ul.\ \L ojasiewicza 6, PL-30348 Kra\-k\'ow, Poland}
   \email{Zenon.Jablonski@im.uj.edu.pl}
   \author[I.\ B.\ Jung]{Il Bong Jung}
   \address{Department of Mathematics, Kyungpook National
University, Daegu 702-701, Korea}
   \email{ibjung@knu.ac.kr}
   \author[J.\ Stochel]{Jan Stochel}
\address{Instytut Matematyki, Uniwersytet Jagiello\'nski,
ul.\ \L ojasiewicza 6, PL-30348 Kra\-k\'ow, Poland}
   \email{Jan.Stochel@im.uj.edu.pl}
   \thanks{Research of the first
and the third authors was supported by the MNiSzW
(Ministry of Science and Higher Education) grant NN201
546438 (2010-2013). The second author was supported by
Basic Science Research Program through the National
Research Foundation of Korea (NRF) grant funded by the
Korea government (MEST) (2009-0093125).}

   \begin{abstract}
An absolute continuity approach to quasinormality
which relates the operator in question to the spectral
measure of its modulus is developed. Algebraic
characterizations of some classes of operators that
emerged in this context are invented. Various examples
and counterexamples illustrating the concepts of the
paper are constructed by means of weighted shifts on
directed trees. Generalizations of these results that
cover the case of $q$-quasinormal operators are
established.
   \end{abstract}
   \maketitle

   \section{Introduction}
The notion of a quasinormal operator, i.e., an
operator with commuting factors in its polar
decomposition, has been introduced by Arlen Brown in
\cite{bro} (the unbounded case has been undertaken in
\cite{StSz1}). Such operators form a bridge between
normal and subnormal operators. Quasinormal operators
have been found to be useful in many constructions of
operator theory, e.g., when dealing with the question
of subnormality (see \cite{Em,lam0,StSz1} for the
general case and \cite{lam1} for the case of
composition operators).

In the present paper we develop an absolute continuity
approach to quasinormality of unbounded operators. On
the way we characterize wider classes of operators
that seem to be of independent interest. First we
prove that a closed densely defined Hilbert space
operator $A$ is quasinormal if and only if
$\is{E(\cdot)Af}{Af} \ll \is{E(\cdot)|A|f}{|A|f}$ for
every vector $f$ in the domain $\dz{A}$ of $A$ (the
symbol $\ll$ means absolute continuity), where $E$ is
the spectral measure of the modulus $|A|$ of $A$ (see
Theorem \ref{chq2}). One may ask whether reversing the
above absolute continuity implies the quasinormality
of $A$. In general the answer is in the negative (cf.\
Examples \ref{nq-ineq}, \ref{nq-ineq2} and
\ref{nq-ineq3}). Another question which arises is:\
assuming more, namely that the Radon-Nikodym
derivative of $\is{E(\cdot)|A|f}{|A|f}$ with respect
to $\is{E(\cdot)Af}{Af}$ is bounded by a constant $c$
which does not depend on $f \in \dz{A}$, is it true
that $A$ is quasinormal? We shall prove that the
answer is in the affirmative for $c \Le 1$ and in the
negative for $c> 1$ (cf.\ Theorem \ref{quasiwkwkw2}
and Section \ref{Exam}). The case of $c>1$ leads to a
new class of operators, called weakly quasinormal,
which are characterized by means of strong commutant
of their moduli (cf.\ Theorem \ref{chq3}). Let us
remark that operators which satisfy the reversed
absolute continuity condition can be completely
characterized in the language of operator theory (cf.\
Theorem \ref{abcon}).

The absolute continuity approach invented in this
paper is implemented in the context of weighted shifts
on directed trees, cf.\ Theorem \ref{wsquasineq} (the
concept of a weighted shift on a directed tree has
been developed in \cite{j-j-s}). This enables us to
illustrate the theme of this article by various
examples and to show that there is no relationship
between the hyponormality class and the classes of
operators being considered in this paper (cf.\ Section
\ref{Exam}). Finally, the last section provides some
generalizations of our main results. In particular,
they cover the case of the so-called $q$-quasinormal
operators.
   \section{Notation and terminology} In what follows,
$\cbb$ stands for the set of all complex numbers. We
denote by $\zbb_+$, $\nbb$ and $\rbb_+$ the sets of
nonnegative integers, positive integers and
nonnegative real numbers, respectively. The symbol
$\borel{\rbb_+}$ stands for the $\sigma$-algebra of
all Borel subsets of $\rbb_+$. Given two finite
positive Borel measures $\mu$ and $\nu$ on $\rbb_+$,
we write $\mu \ll \nu$ if $\mu$ is absolutely
continuous with respect to $\nu$; if this is the case,
then $\D\mu/\D\nu$ stands for the Radon-Nikodym
derivative of $\mu$ with respect to $\nu$. We denote
by $\chi_Y$ the characteristic function of a set $Y$.
The symbol $\bigsqcup$ is reserved to denote pairwise
disjoint union of sets.

Let $A$ be an operator in a complex Hilbert space
$\hh$ (all operators considered in this paper are
assumed to be linear). Denote by $\dz{A}$, $\ob{A}$,
$A^*$ and $\bar A$ the domain, the range, the adjoint
and the closure of $A$ (in case they exist). If $A$ is
closed and densely defined, then $|A|$ stands for the
square root of the positive selfadjoint operator
$A^*A$ (for this and other necessary facts concerning
unbounded operators we refer the reader to
\cite{b-s,weid}). For two operators $S$ and $T$ in
$\hh$, we write $S\subseteq T$ if $\dz{S} \subseteq
\dz{T}$ and $Sf=Tf$ for all $f \in \dz{S}$. The
$C^*$-algebra of all bounded operators $A$ in $\hh$
such that $\dz{A}=\hh$ is denoted by $\ogr \hh$. The
symbol $I_{\hh}$ stands for the identity operator on
$\hh$. We write $\lin \ff$ for the linear span of a
subset $\ff$ of $\hh$.

We now recall a description of the strong commutant of
a normal operator.
   \begin{thm}[\mbox{\cite[Theorem 6.6.3]{b-s}}] \label{bir}
Let $A$ be a normal operator in $\hh$, i.e., $A$ is
closed densely defined and $A^*A=AA^*$. If $T\in
\ogr{\hh}$, then $TA \subseteq AT$ if and only if $T$
commutes with the spectral measure of $A$.
   \end{thm}

A densely defined operator $S$ in $\hh$ is said to be
{\em subnormal} if there exists a complex Hilbert
space $\kk$ and a normal operator $N$ in $\kk$ such
that $\hh \subseteq \kk$ (isometric embedding) and $Sh
= Nh$ for all $h \in \dz S$. A densely defined
operator $T$ in $\hh$ is said to be {\em hyponormal}
if $\dz{T} \subseteq \dz{T^*}$ and $\|T^*f\| \Le
\|Tf\|$ for all $f \in \dz T$. It is well-known that
subnormal operators are hyponormal, but not
conversely. Recall that subnormal (hyponormal)
operators are closable and their closures are
subnormal (hyponormal). We refer the reader to
\cite{ot-sch,jj3,sto} and
\cite{StSz0,StSz1,StSz1.5,StSz2} for elements of the
theory of unbounded hyponormal and subnormal
operators, respectively.
   \section{An absolute continuity approach to quasinormality}
Following \cite{StSz1} (see also \cite{bro} for the
case of bounded operators), we say that a closed
densely defined operator $A$ in a complex Hilbert
space $\hh$ is {\em quasinormal} if $A$ {\em commutes}
with the spectral measure $E$ of $|A|$, i.e.,
$E(\sigma)A \subseteq A E(\sigma)$ for all $\sigma \in
\borel{\rbb_+}$. The ensuing fact is well-known (use
\cite[Proposition 1]{StSz1} and Theorem~ \ref{bir}).
   \begin{align} \label{quas}
   \begin{minipage}{70ex}
A closed densely defined operator $A$ in $\hh$ is
quasinormal if and only if $U |A| \subseteq |A|U$,
where $A=U|A|$ is the polar decomposition of $A$.
   \end{minipage}
   \end{align}
Note that quasinormal operators are hyponormal
(indeed, since $A \subseteq |A|U$, we get $U^*|A|
\subseteq A^*$, which implies hyponormality). In fact,
quasinormal operators are always subnormal (see
\cite[Theorem 2]{StSz1} for the general case; the
bounded case can be deduced from \cite[Theorem
1]{bro}). The reverse implication does not hold. For
more information on quasinormal operators we refer the
reader to \cite{bro,con} (bounded operators) and
\cite{StSz1,maj} (unbounded operators).

Now we show that quasinormality can be characterized
by means of absolute continuity.
   \begin{thm}\label{chq2}
Let $A$ be a closed densely defined operator in $\hh$
and $E$ be the spectral measure of $|A|$. Then the
following three conditions are equivalent\/{\em :}
   \begin{enumerate}
   \item[(i)] $A$ is quasinormal,
   \item[(ii)]
$\is{E(\sigma)Af}{Af} = \is{E(\sigma)|A|f}{|A|f}$ for all
$\sigma \in \borel{\rbb_+}$ and $f \in \dz{A}$,
   \item[(iii)]
$\is{E(\cdot)Af}{Af} \ll \is{E(\cdot)|A|f}{|A|f}$ for
every $f \in \dz{A}$.
   \end{enumerate}
   \end{thm}
   \begin{proof}
   (i)$\Rightarrow$(ii) Let $A = U |A|$ be the polar
decomposition of $A$. By \eqref{quas} and Theorem
\ref{bir}, we have $U E(\cdot) = E(\cdot)U$. Since
$P:=U^*U$ is the orthogonal projection of $\hh$ onto
$\overline{\ob{|A|}}$, we see that $P|A|=|A|$.
Combining all this together, we get
   \begin{align*}
\is{E(\sigma)Af}{Af} = \is{E(\sigma)|A|f}{P|A|f} =
\is{E(\sigma)|A|f}{|A|f}, \quad \sigma \in
\borel{\rbb_+}, \, f \in \dz{A}.
   \end{align*}

   (ii)$\Rightarrow$(iii) Obvious.

   (iii)$\Rightarrow$(i) Fix finite systems $\sigma_1,
\ldots, \sigma_n \in \borel{\rbb_+}$ and $f_1, \ldots,
f_n \in \dz{A}$. Then there exist finite systems
$\sigma^{\prime}_1, \ldots, \sigma^{\prime}_m \in
\borel{\rbb_+}$ and $J_1, \ldots, J_n \subseteq \{1,
\ldots,m\}$ such that $\sigma^{\prime}_k \cap
\sigma^{\prime}_l = \varnothing$ for all $k \neq l$,
and $\sigma_i = \bigcup_{j \in J_i} \sigma^{\prime}_j$
for all $i\in \{1, \ldots, n\}$. Set $f^{\prime}_j =
\sum_{i=1}^n \chi_{J_i} (j) f_i$ for $j \in \{1,
\ldots,m\}$. It is easily seen that
   \allowdisplaybreaks
   \begin{align} \notag
\Big\|\sum_{i=1}^n E(\sigma_i)Af_i\Big\|^2 & =
\Big\|\sum_{i=1}^n \sum_{j \in J_i} E(\sigma^{\prime}_j) A
f_i\Big\|^2
   \\    \notag
&= \Big\|\sum_{i=1}^n \sum_{j=1}^m \chi_{J_i} (j)
E(\sigma^{\prime}_j) Af_i\Big\|^2
   \\    \notag
&= \Big\|\sum_{j=1}^m E(\sigma^{\prime}_j)
A\Big(\sum_{i=1}^n \chi_{J_i} (j) f_i\Big)\Big\|^2
   \\    \label{r6}
&= \sum_{j=1}^m
\is{E(\sigma^{\prime}_j)Af^{\prime}_j}{Af^{\prime}_j}.
   \end{align}
   Arguing as above, we get
   \begin{align} \label{r7}
\Big\|\sum_{i=1}^n E(\sigma_i)|A| f_i\Big\|^2 =
\sum_{j=1}^m \is{E(\sigma^{\prime}_j)|A| f^{\prime}_j}{|A|
f^{\prime}_j}.
   \end{align}
Since $E(\sigma)|A| \subseteq |A| E(\sigma)$ for all
$\sigma \in \borel{\rbb_+}$, we have
   \begin{align}    \label{oba}
\ob{|A|} = \lin\big\{E(\sigma)|A| f \colon \sigma \in
\borel{\rbb_+}, f \in \dz{A}\big\}.
   \end{align}
Combining \eqref{r6} and \eqref{r7} with (iii), we
deduce that for all finite systems $\sigma_1, \ldots,
\sigma_n \in \borel{\rbb_+}$ and $f_1, \ldots, f_n \in
\dz{A}$ the following implication holds:
   \begin{align*}
\sum_{i=1}^n E(\sigma_i)|A|f_i = 0 \implies
\sum_{i=1}^n E(\sigma_i)Af_i = 0.
   \end{align*}
This together with \eqref{oba} implies that the map
$\widetilde T_0\colon \ob{|A|} \to \hh$ given by
   \begin{align}     \label{defto}
\widetilde T_0\bigg(\sum_{i=1}^n
E(\sigma_i)|A|f_i\bigg) = \sum_{i=1}^n
E(\sigma_i)Af_i, \quad \sigma_i \in \borel{\rbb_+}, \,
f_i\in \dz{A}, n \in \nbb,
   \end{align}
is well-defined and linear. Substituting $n=1$ and
$\sigma_1=\rbb_+$ into \eqref{defto}, we see that
$\widetilde T_0|A| = A$. This yields
   \begin{align*}
\big\|\widetilde T_0(|A|f)\big\| = \big\|Af\big\| =
\big\||A|f\big\|, \quad f \in \dz{A},
   \end{align*}
which means that $\widetilde T_0$ is an isometry. Let
$T_0\colon \overline{\ob{|A|}} \to \hh$ be a unique
isometric and linear extension of $\widetilde T_0$. Define
the operator $T\in \ogr{\hh}$ by $Tf = T_0 Pf$ for $f \in
\hh$, where $P\in \ogr{\hh}$ is the orthogonal projection
of $\hh$ onto $\overline{\ob{|A|}}$. In view of
\eqref{defto}, $T$ is an extension of $T_0$ such that
   \begin{align} \label{tema1}
A & = T|A|,
   \\
TE(\sigma)|A| & = E(\sigma)A, \quad \sigma \in
\borel{\rbb_+}. \label{tema}
   \end{align}
Since $T_0$ is an isometry, we infer from the
definition of $T$ that $\jd{T} = \jd{|A|}=\jd{A}$.
This implies that $T$ is a partial isometry and $A =
T|A|$ is the polar decomposition of $A$. Since
   \begin{align*}
TE(\sigma)(|A|f) \overset{\eqref{tema}}= E(\sigma) A f
\overset{\eqref{tema1}}= E(\sigma) T(|A|f), \quad \sigma
\in \borel{\rbb_+}, \, f \in \dz{A},
   \end{align*}
we deduce that $TE(\sigma)|_{\overline{\ob{|A|}}} =
E(\sigma) T|_{\overline{\ob{|A|}}}$ for all $\sigma
\in \borel{\rbb_+}$. As $\jd{|A|}$ reduces $E$ and
$T|_{\jd{|A|}}=0$, we conclude that $T$ commutes with
the spectral measure $E$ of $|A|$. By Theorem
\ref{bir}, we have $T|A| \subseteq |A|T$, which
together with \eqref{quas} completes the proof.
   \end{proof}
   \section{\label{bbcc}A characterization of weak quasinormality}
We say that a closed densely defined operator $A$ in a
complex Hilbert space $\hh$ is {\em weakly
quasinormal} if there exists $c\in \rbb_+$ such that
   \begin{align}  \label{inc1}
\is{E(\sigma)|A|f}{|A|f} \Le c \is{E(\sigma)Af}{Af},
\quad \sigma \in \borel{\rbb_+}, \, f \in \dz{A},
   \end{align}
where $E$ is the spectral measure of $|A|$ (or
equivalently:\ for every $f\in \dz{A}$,
$\is{E(\cdot)|A|f}{|A|f} \ll \is{E(\cdot)Af}{Af}$ and
$\D \is{E(\cdot)|A|f}{|A|f}/\D \is{E(\cdot)Af}{Af} \Le
c$ almost everywhere with respect to
$\is{E(\cdot)Af}{Af}$). The smallest such $c$ will be
denoted by $\cfr_A$. It is worth mentioning that the
constant $\cfr_A$ is always greater than or equal to
$1$ whenever the operator $A$ is nonzero. As proved in
Theorem \ref{quasiwkwkw2}, a nonzero closed and
densely defined operator $A$ is quasinormal if and
only if it is weakly quasinormal with $\cfr_A=1$.

Our goal in this section is to characterize weak
quasinormality of unbounded operators. We begin with a
technical lemma.
   \begin{lem} \label{izonp}
Let $T \in\ogr{\hh}$ be a contraction whose
restriction to a closed linear subspace $\kk$ of $\hh$
is isometric. Then $T^*Tk =k$ for all $k \in \kk$.
   \end{lem}
   \begin{proof}
Denote by $T|_{\kk}\colon \kk \to \hh$ the restriction
of $T$ to $\kk$. Since $(T|_{\kk})^* = P_{\kk}T^*$ and
$T|_{\kk}$ is an isometry, we have $P_{\kk}T^*
T|_{\kk} = I_{\kk}$. This and $\|T\| \Le 1$ yield
   \begin{align*}
\|k\| = \|P_{\kk}T^* T k\| \Le \|T^* T k\| \Le \|k\|,
\quad k\in \kk,
   \end{align*}
which implies that $T^*Tk \in \kk$ and thus $T^*Tk =
k$ for all $k \in \kk$.
   \end{proof}
Below we show that Lemma \ref{izonp} is not true if
$T$ is not a contraction.
   \begin{exa}
Let $\kk$ be a nonzero complex Hilbert space and let
$\hh:=\kk\oplus \kk$. Take $T = \left
[\begin{smallmatrix} A & B
\\ C & D \end{smallmatrix}\right] \in
\ogr{\hh}$ with $A, B, C, D \in \ogr {\kk}$. Then
$T|_{\kk \oplus \{0\}}$ is an isometry if and only if
$A^*A + C^*C=I_{\kk}$. It is also easily seen that
$T^*T(\kk \oplus \{0\}) \subseteq \kk \oplus \{0\}$ if
and only if $B^*A + D^*C = 0$. Substituting $\kk=\cbb$
and $A=C=\frac{1}{\sqrt{2}}$, and taking $B, D \in
\cbb$ such that $B+D\neq 0$, we see that $T|_{\kk
\oplus \{0\}}$ is an isometry and $T^*T(\kk \oplus
\{0\}) \not\subseteq \kk \oplus \{0\}$. Hence, by
Lemma \ref{izonp}, $T$ is not a contraction,
independently of whether $B$ and $D$ are large or
small numbers.
   \end{exa}
   Now we are ready to characterize weak
quasinormality. We will show in Section \ref{Exam}
that there exist weakly quasinormal operators $A$ with
$\cfr_A > 1$ which are not hyponormal (and thus not
quasinormal).
   \begin{thm}\label{chq3}
Let $A$ be a closed densely defined operator in $\hh$
and let $c\in \rbb_+$. Then the following two
conditions are equivalent\/{\em :}
   \begin{enumerate}
   \item[(i)] $A$ is weakly quasinormal with $\cfr_A \Le c$,
   \item[(ii)] there exists $T \in \ogr{\hh}$ such that
   \begin{align} \label{comtn}
\text{$TA=|A|$, $T|A|\subseteq |A|T$ and $\|T\| \Le
\sqrt{c}$.}
   \end{align}
   \end{enumerate}
Moreover, the following assertions are valid{\em :}
   \begin{enumerate}
   \item[(iii)] if $A$ is weakly
quasinormal and $A\neq 0$, then $\cfr_A\Ge 1$,
   \item[(iv)] if $A$ is weakly
quasinormal, then the operator $T$ in {\em (ii)} can
be chosen so that $\ob{T}=\overline{\ob{|A|}}$ and
$\|T\|=\sqrt{\cfr_A}$,
   \item[(v)] if $T \in \ogr{\hh}$ satisfies
\eqref{comtn}, then $T|_{\overline{\ob{A}}} \colon
\overline{\ob{A}} \to \hh$ is an isometry,
$T(\overline{\ob{A}}) = \overline{\ob{|A|}}$ and the
partial isometry $U$ in the polar decomposition of $A$
takes the form $U=PT^*$, where $P \in \ogr{\hh}$ is
the orthogonal projection of $\hh$ onto
$\overline{\ob{A}}$.
   \end{enumerate}
   \end{thm}
   \begin{proof}
Let $E$ be the spectral measure of $|A|$.

(i)$\Rightarrow$(ii) Without loss of generality we can
assume that $c=\cfr_A$. Arguing as in the proof of
implication (iii)$\Rightarrow$(i) of Theorem
\ref{chq2}, we deduce that for all finite systems
$\sigma_1, \ldots, \sigma_n \in \borel{\rbb_+}$ and
$f_1, \ldots, f_n \in \dz{A}$,
   \begin{align}   \label{ineqc}
\Big\|\sum_{i=1}^n E(\sigma_i)|A| f_i\Big\|^2 \Le c
\Big\|\sum_{i=1}^n E(\sigma_i)A f_i\Big\|^2.
   \end{align}
Define the closed vector space $\hh_0$ by
   \begin{align*}
\hh_0 = \overline{\lin\big\{E(\sigma)Af\colon \sigma
\in \borel{\rbb_+}, \, f \in \dz{A}\big\}}.
   \end{align*}
It follows from \eqref{ineqc} that there exists a
unique bounded linear map $T_0 \colon \hh_0 \to \hh$
such that $\|T_0\|\Le \sqrt{c}$ and
   \begin{align} \label{toes}
T_0E(\sigma) A = E(\sigma) |A|, \quad \sigma \in
\borel{\rbb_+}.
   \end{align}
Define the operator $T \in \ogr{\hh}$ by $Tf=T_0Qf$
for $f \in \hh$, where $Q\in \ogr{\hh}$ is the
orthogonal projection of $\hh$ onto $\hh_0$. Then $T$
is an extension of $T_0$ such that $\|T\|\Le
\sqrt{c}$. Substituting $\sigma=\rbb_+$ into
\eqref{toes}, we get
   \begin{align}   \label{tam}
TA = |A|.
   \end{align}
Applying the equation \eqref{toes} twice yields
   \begin{multline} \label{numerek}
TE(\sigma)(E(\tau)Af) = TE(\sigma \cap \tau)Af =
E(\sigma \cap \tau)|A|f
   \\
= E(\sigma) E(\tau)|A|f = E(\sigma)T (E(\tau)Af), \quad
f \in \dz{A}, \, \sigma, \tau \in \borel{\rbb_+}.
   \end{multline}
   Hence $TE(\sigma)|_{\hh_0} = E(\sigma)T|_{\hh_0}$
for all $\sigma \in \borel{\rbb_+}$. Since $\hh_0$
reduces the spectral measure $E$ and $T|_{\hh \ominus
\hh_0} = 0$, we obtain $TE(\cdot)=E(\cdot)T$. Applying
Theorem \ref{bir}, we get $T|A|\subseteq |A|T$. By
\eqref{toes}, the definition of $T$ and $E(\cdot)|A|
\subseteq |A| E(\cdot)$, we have
   \begin{align} \label{alej}
\ob{T} = \ob{T_0} \subseteq \overline{\ob{|A|}}.
   \end{align}
It follows from \eqref{tam} that
   \begin{align} \label{cos}
\|T(Af)\| = \big\||A|f\big\| = \|Af\|, \quad f \in
\dz{A}.
   \end{align}
Thus the operator $T|_{\overline{\ob{A}}} \colon
\overline{\ob{A}} \to \hh$ is an isometry. Since
$TA=|A|$, we see that
   \begin{align*}
\overline{\ob{|A|}}=T(\overline{\ob{A}})\subseteq
\ob{T} \overset{\eqref{alej}} \subseteq
\overline{\ob{|A|}},
   \end{align*}
which means that $\ob{T} = \overline{\ob{|A|}}$.

(ii)$\Rightarrow$(i) Let $P \in \ogr{\hh}$ be the
orthogonal projection of $\hh$ onto
$\overline{\ob{A}}$ and let $U:=PT^*$. If $h \in \hh$,
then
   \allowdisplaybreaks
   \begin{align*}
h \in \jd{U} & \iff \is{PT^*h}{Af} = 0 \quad \text{for
all } f\in \dz{A}
   \\
& \iff \is{h}{TAf} = 0 \quad \text{for all } f \in
\dz{A}
   \\
& \overset{\eqref{comtn}} \iff \is{h}{|A|f} = 0 \quad
\text{for all } f \in \dz{A}
   \\
& \iff h \in \hh \ominus \overline{\ob{|A|}} =
\jd{|A|}=\jd{A},
   \end{align*}
which shows that $\jd{U}=\jd{A}$. Using the equality
$TA=|A|$ and arguing as in \eqref{cos}, we see that
the operator $T|_{\overline{\ob{A}}} \colon
\overline{\ob{A}} \to \hh$ is an isometry and
$T(\overline{\ob{A}}) = \overline{\ob{|A|}}$. This
implies that $I_{\overline{\ob{A}}} =
PT^*T|_{\overline{\ob{A}}}$ (see the proof of Lemma
\ref{izonp}). Thus
   \begin{align*}
A = PT^*T A \overset{ \eqref{comtn}}= PT^* |A| = U
|A|.
   \end{align*}
This and the equalities $\jd{U}=\jd{A}=\jd{|A|}$ imply
that $A=U|A|$ is the polar decomposition of $A$, which
proves (v).

It follows from \eqref{comtn} and Theorem \ref{bir}
that
   \begin{align} \notag
\is{E(\sigma)|A|f}{|A|f} & = \is{E(\sigma)T A f}{T A
f} = \big\|T E(\sigma) A f\big\|^2
   \\
&\Le c \is{E(\sigma)Af}{Af}, \quad \sigma \in
\borel{\rbb_+},\, f \in \dz{A}, \label{numerek1}
   \end{align}
which shows that $A$ is weakly quasinormal and $\cfr_A
\Le c$.

(iii) If $A\neq 0$ satisfies \eqref{inc1}, then
substituting $\sigma=\rbb_+$ into \eqref{inc1} yields
   \begin{align*}
\||A|f\|^2 \Le c \|Af\|^2 = c\||A|f\|^2, \quad f \in
\dz{A},
   \end{align*}
which gives (iii).

It only remains to prove (iv). Assume that (i) holds.
It follows from the proof of the implication
(i)$\Rightarrow$(ii) that there exists $T\in
\ogr{\hh}$ such that \eqref{comtn} holds with
$c=\cfr_A$, i.e.\ $\|T\|\Le \sqrt{\cfr_A}$, and
$\ob{T} = \overline{\ob{|A|}}$. Now applying the
reverse implication (ii)$\Rightarrow$(i) with
$c=\|T\|^2$, we get $\cfr_A \Le \|T\|^2$. This
completes the proof.
   \end{proof}
Reversing the absolute continuity in Theorem
\ref{chq2}(iii) leads to a new class of operators that
is essentially wider than the class of weakly
quasinormal operators (cf.\ Section \ref{Exam}). The
new class can be characterized as follows.
   \begin{thm}\label{abcon}
Let $A$ be a closed densely defined operator in $\hh$
and $E$ be the spectral measure of $|A|$. Then the
following two conditions are equivalent\/{\em :}
   \begin{enumerate}
   \item[(i)] $\is{E(\cdot)|A|f}{|A|f} \ll
\is{E(\cdot)Af}{Af}$ for every $f \in \dz{A}$,
   \item[(ii)] there exists a $($unique\/$)$ linear
map $T_0 \colon \hscr_0 \to \ob{|A|}$ such
that\footnote{\;Note that $\ob{A} \subseteq \hscr_0$
and $E(\sigma)\hscr_0 \subseteq \hscr_0$ for all
$\sigma \in \borel{\rbb_+}$.} $T_0 A=|A|$ and
$T_0E(\sigma)|_{\hscr_0}=E(\sigma)T_0$ for all $\sigma
\in \borel{\rbb_+}$, where
   \begin{align*}
\hscr_0 = \lin\big\{E(\sigma)Af \colon \sigma \in
\borel{\rbb_+ }, \, f \in \dz{A}\big\}.
   \end{align*}
   \end{enumerate}
Moreover, the following assertion holds for any $c \in
\rbb_+${\em :}
\begin{enumerate}
 \item[(iii)] $A$ is weakly
quasinormal with $\cfr_A \Le c$ if and only if $T_0$
is bounded and $\|T_0\| \Le \sqrt{c}$, where $T_0$ is
as in {\em (ii)}.
   \end{enumerate}
   \end{thm}
   \begin{proof}
(i)$\Rightarrow$(ii) Arguing as in the proof of the
implication (iii)$\Rightarrow$(i) of Theorem
\ref{chq2}, we show that the following implication
holds for all finite systems $\sigma_1, \ldots,
\sigma_n \in \borel{\rbb_+}$ and $f_1, \ldots, f_n \in
\dz{A}$,
   \begin{align*}
\sum_{i=1}^n E(\sigma_i)Af_i = 0 \quad \implies \quad
\sum_{i=1}^n E(\sigma_i)|A|f_i = 0.
   \end{align*}
This, combined with the fact that $E(\sigma)|A|
\subseteq |A| E(\sigma)$ for all $\sigma \in
\borel{\rbb_+}$, implies that the map $T_0\colon
\hscr_0\to \ob{|A|}$ given by
   \begin{align}     \label{defto+}
T_0\Big(\sum_{i=1}^n E(\sigma_i)Af_i\Big) = \sum_{i=1}^n
E(\sigma_i)|A|f_i, \quad \sigma_i \in \borel{\rbb_+}, \,
f_i\in \dz{A}, n \in \nbb,
   \end{align}
is well-defined and linear. Substituting $n=1$ and
$\sigma_1=\rbb_+$ into \eqref{defto+}, we see that
$T_0 A = |A|$. Arguing as in \eqref{numerek} with
$T_0$ in place of $T$, we verify that
$T_0E(\sigma)|_{\hscr_0}=E(\sigma)T_0|_{\hscr_0}$ for
all $\sigma \in \borel{\rbb_+}$.

It is clear that any linear map $T_0 \colon \hscr_0
\to \ob{|A|}$ with the properties specified by (ii)
must satisfy \eqref{defto+}, and as such is unique.

(ii)$\Rightarrow$(i) If $f \in \dz{A}$ and $\sigma \in
\borel{\rbb_+}$ are such that $\is{E(\sigma)Af}{Af}=0$,
then $E(\sigma) A f=0$, and thus
   \begin{align*}
\is{E(\sigma)|A|f}{|A|f} = \is{E(\sigma)T_0 A f}{T_0 A
f} = \big\|T_0 E(\sigma) A f\big\|^2 =0,
   \end{align*}
which gives (i).

Now we justify the ``moreover'' part of the
conclusion. If $A$ is weakly quasinormal with $\cfr_A
\Le c$ and $T\in\ogr{\hh}$ is as in Theorem
\ref{chq3}\,(ii), then clearly (use Theorem \ref{bir})
$T|_{\hscr_0}=T_0$, which implies the boundedness of
$T_0$ and gives $\|T_0\| \Le \sqrt{c}$. In turn, if
$T_0$ is bounded and $\|T_0\| \Le \sqrt{c}$, then by
mimicking the argument used in \eqref{numerek1} with
$T_0$ in place of $T$, we see that $A$ is weakly
quasinormal with $\cfr_A \Le c$. This completes the
proof.
   \end{proof}
   \begin{rem}
It is worth mentioning that if $A$ is a closed densely
defined operator in $\hh$ which is not weakly
quasinormal and which satisfies the condition (i) of
Theorem \ref{abcon} (see Section \ref{Exam} for how to
construct such operators), then the operator $T_0$
appearing in the condition (ii) of Theorem \ref{abcon}
is unbounded and it extends the isometric operator
$T_0|_{\ob{A}}$ (for the latter, consult \eqref{cos}).
   \end{rem}
   \section{Quasinormality revisited}
   In this short section we show that quasinormality
is completely characterized by the inequality
\eqref{inc1} with $c=1$. Recall that if $A$ is a
nonzero weakly quasinormal operator, then $\cfr_A \Ge
1$ (see Theorem \ref{chq3}\,(iii)).
   \begin{thm} \label{quasiwkwkw2}
Let $A$ be a nonzero closed densely defined operator
in $\hh$. Then the following two conditions are
equivalent\/{\em :}
   \begin{enumerate}
   \item[(i)] $A$ is quasinormal,
   \item[(ii)] $A$ is weakly quasinormal with $\cfr_A=1$.
   \end{enumerate}
   \end{thm}
   \begin{proof}
(i)$\Rightarrow$(ii) Apply Theorem \ref{chq2}.

(ii)$\Rightarrow$(i) By Theorem \ref{chq3} there
exists an operator $T \in \ogr{\hh}$ that satisfies
\eqref{comtn} with $c=1$. Then, by the assertion (v)
of Theorem \ref{chq3}, the operator
$T|_{\overline{\ob{A}}} \colon \overline{\ob{A}} \to
\hh$ is an isometry. Since $\|T\| \Le 1$, we infer
from Lemma \ref{izonp} that
$I_{\overline{\ob{A}}}=T^*T|_{\overline{\ob{A}}}$.
This and the equality $TA=|A|$ yield
   \begin{align} \label{atta}
A=T^*TA = T^* |A|.
   \end{align}
By \eqref{comtn} and Theorem \ref{bir}, the operator
$T^*$ commutes with the spectral measure $E$ of $|A|$.
Hence, by using the fact hat $E(\cdot)|A| \subseteq
|A| E(\cdot)$, we deduce that
   \begin{align*}
E(\sigma)A \overset{\eqref{atta}} = E(\sigma) T^*|A| =
T^* E(\sigma) |A| \subseteq T^* |A| E(\sigma)
\overset{\eqref{atta}} = A E(\sigma)
   \end{align*}
for all $\sigma \in \borel{\rbb_+}$, which means that
$A$ is quasinormal.
   \end{proof}
 \section{Boundedness of weakly
quasinormal operators} Taking a quick look at
Corollary \ref{sleu} would suggest that if $A$ is a
quasinormal operator, then $\ob{A} \subseteq
\dz{|A|^\alpha}$ for every positive real number
$\alpha$. However, as shown below, this is not
necessarily the case. In fact, if $A$ is an unbounded
quasinormal operator, then $\ob{A} \not\subseteq
\dz{|A|^\alpha}$ for every positive real number
$\alpha$. The particular case of $\alpha \Ge 1$ can be
deduced from \cite[Lemma A.1]{Se-St} and \cite[Theorem
3.3]{ota} (because $\dz{|A|^{\alpha}} \subseteq
\dz{|A|} = \dz{A} \subseteq \dz{A^*}$ for $\alpha \Ge
1$). It is worth pointing out that there are unbounded
closed densely defined Hilbert space operators $A$
such that $\ob{A}\subseteq \dz{|A|}$ (cf.\
\cite{ota}).
   \begin{pro}\label{qinvb}
If $A$ is a weakly quasinormal operator in $\hh$ such
that $\ob{A} \subseteq \dz{|A|^\alpha}$ for some
positive real number $\alpha$, then $A \in \ogr{\hh}$.
   \end{pro}
   \begin{proof}
First we show that
   \begin{align} \label{dz2=}
\dz{|A|} = \dz{|A|^{1+\alpha}}.
   \end{align}
The inclusion ``$\supseteq$'' is always true (cf.\
\cite[Lemma A.1]{Se-St}). To prove the reverse
inclusion, take $f \in \dz{|A|}$. Denote by $E$ the
spectral measure of $|A|$. Since $A$ is weakly
quasinormal, we get
   \begin{align*}
\int_\sigma x^2 \is{E(\D x)f}{f} =
\is{E(\sigma)|A|f}{|A|f} \Le \cfr_A
\is{E(\sigma)Af}{Af}, \quad \sigma \in \borel{\rbb_+}.
   \end{align*}
This and the assumption $\ob{A} \subseteq
\dz{|A|^\alpha}$ imply that
   \begin{align*}
\int_0^\infty x^{2(1+\alpha)} \is{E(\D x)f}{f} &=
\int_0^\infty x^{2\alpha} \is{E(\D x)|A|f}{|A|f}
   \\
&\Le \cfr_A \int_0^\infty x^{2\alpha} \is{E(\D
x)Af}{Af} = \cfr_A \big\||A|^\alpha Af\big\|^2 <
\infty.
   \end{align*}
Hence $f \in \dz{|A|^{1+\alpha}}$, which completes the
proof of \eqref{dz2=}. Now, by applying \cite[Lemma
A.1]{Se-St} to \eqref{dz2=}, we conclude that $|A| \in
\ogr{\hh}$. This in turn implies that $A \in
\ogr{\hh}$, which completes the proof.
   \end{proof}
   \section{Weakly quasinormal weighted shifts on directed trees}
The basic facts on directed trees and weighted shifts
on directed trees can be found in \cite{j-j-s}. We
refer the reader to \cite{j-j-s2} for the recent
applications of this idea to general operator theory.

Let $\tcal=(V,E)$ be a directed tree ($V$ and $E$
stand for the sets of vertices and edges of $\tcal$,
respectively). Set $V^\circ=V\setminus \{\koo\}$ if
$\tcal$ has a root and $V^\circ=V$ otherwise. For
every vertex $u \in V^\circ$ there exists a unique
vertex, denoted by $\pa u$, such that $(\pa u,u)\in
E$. Set $\dzi u = \{v\in V\colon (u,v)\in E\}$ for $u
\in V$. If $W \subseteq V$, we put $\dzi W =
\bigcup_{v \in W} \dzi v$ and $\des W =
\bigcup_{n=0}^\infty \dzin n W$, where $\dzin{0}{W} =
W$ and $\dzin{n+1}{W} = \dzi{\dzin{n}{W}}$ for all
integers $n\Ge 0$. For $u \in V$, we set $\dzin n
u=\dzin n {\{u\}}$ and $\des{u}=\des{\{u\}}$.

Denote by $\ell^2(V)$ the Hilbert space of all square
summable complex functions on $V$ with the standard
inner product. The set $\{e_u\}_{u\in V}$, where $e_u
:= \chi_{\{u\}}$, is an orthonormal basis of
$\ell^2(V)$. Put $\escr = \lin \{e_u\colon u \in V\}$.

Given $\lambdab = \{\lambda_v\}_{v \in V^\circ}
\subseteq \cbb$, we define the operator $\slam$ in
$\ell^2(V)$ by
   \begin{align*}
   \begin{aligned}
\dz {\slam} & = \{f \in \ell^2(V) \colon
\varLambda_\tcal f \in \ell^2(V)\},
   \\
\slam f & = \varLambda_\tcal f, \quad f \in \dz
{\slam},
   \end{aligned}
   \end{align*}
where $\varLambda_\tcal$ is the map defined on
functions $f\colon V \to \cbb$ via
   \begin{align} \label{lamtauf}
(\varLambda_\tcal f) (v) =
   \begin{cases}
\lambda_v \cdot f\big(\pa v\big) & \text{ if } v\in
V^\circ,
   \\
0 & \text{ if } v=\koo.
   \end{cases}
   \end{align}
Such $\slam$ is called a {\em weighted shift} on the
directed tree $\tcal$ with weights $\{\lambda_v\}_{v
\in V^\circ}$. Let us recall that weighted shifts on
directed trees are always closed (cf.\
\cite[Proposition 3.1.2]{j-j-s}).

Before characterizing weak quasinormality of a densely
defined weighted shift $\slam$ on $\tcal$, we describe
the spectral measure of $|\slam|^\alpha$ for $\alpha
\in (0,\infty)$.
   \begin{lem} \label{spmm}
If $\slam$ is a densely defined weighted shift on a
directed tree $\tcal$ with weights $\lambdab =
\{\lambda_v\}_{v \in V^\circ}$, $\alpha \in
(0,\infty)$ and $E$ is the spectral measure of
$|\slam|^\alpha$, then
   \begin{align*}
(E (\sigma) f)(v) = \chi_\sigma(\|\slam e_v\|^\alpha)
f(v), \quad v \in V,\, f \in \ell^2(V), \, \sigma \in
\borel{\rbb_+}.
   \end{align*}
   \end{lem}
   \begin{proof}
By \cite[Proposition 3.4.3]{j-j-s}, $\escr \subseteq
\dz{\slam} \cap \dz{|\slam|^\alpha}$ and
$|\slam|^\alpha e_u = \|\slam e_u\|^\alpha e_u$ for
all $u \in V$. Hence, by \cite[(2.2.1)]{j-j-s}, we
have
   \begin{align*}
E (\sigma) f = \sum_{u \in V} \chi_\sigma(\|\slam
e_u\|^\alpha) \is f {e_u} e_u, \quad f \in \ell^2(V),
\, \sigma \in \borel{\rbb_+},
   \end{align*}
which implies that
   \begin{align*}
(E (\sigma) f)(v) = \is{E (\sigma) f}{e_v} =
\chi_\sigma(\|\slam e_v\|^\alpha) f(v), \quad v \in V,
\, f \in \ell^2(V), \, \sigma \in \borel{\rbb_+}.
   \end{align*}
This completes the proof.
   \end{proof}
Now we characterize weak quasinormality of weighted
shifts on directed trees.
   \begin{thm}\label{wsquasineq}
Let $\slam$ be a densely defined weighted shift on a
directed tree $\tcal$ with weights $\lambdab =
\{\lambda_v\}_{v \in V^\circ}$ and $E$ be the spectral
measure of $|\slam|$. Then the following assertions
are valid.
   \begin{enumerate}
   \item[(i)] For any $c\in \rbb_+$,
$\slam$ is weakly quasinormal with $\cfr_{\slam} \Le
c$ if and only if\;\footnote{\;\label{expr}We adhere
to the convention that $\sum_{v\in\varnothing}
|\lambda_v|^2=0$. Note also that $\escr \subseteq
\dz{\slam}$ (cf.\ \cite[Proposition 3.1.3(v)]{j-j-s}),
which means that the expression \eqref{abc2} makes
sense.}
   \begin{gather}  \label{abc2}
\|\slam e_u\|^2 \Le c \sum_{v\in \dzie{u}}
|\lambda_v|^2, \quad u \in V,
   \end{gather}
where $\dzie{u} := \{v \in \dzi{u}\colon \|\slam e_v\|
= \|\slam e_u\|\}$.
   \item[(ii)] $\is{E(\cdot)|\slam|f}{|\slam|f} \ll
\is{E(\cdot)\slam f}{\slam f}$ for all $f \in
\dz{\slam}$ if and only if
   \begin{align} \label{abc3}
\forall u \in V\colon \|\slam e_u\| \neq 0 \implies
\dzip{u} \neq \varnothing,
   \end{align}
where $\dzip{u}:=\{v \in \dzie{u}\colon \lambda_v \neq
0\}$.
   \end{enumerate}
   \end{thm}
It is worth noting that if \eqref{abc2} holds, then
according to our summation convention (see footnote
\ref{expr}) we have
   \begin{align} \label{abc}
\forall u \in V\colon \|\slam e_u\| \neq 0 \implies
\dzie{u} \neq \varnothing.
   \end{align}
   \begin{proof}[Proof of Theorem \ref{wsquasineq}]
It follows from Lemma \ref{spmm} that
   \begin{align*}
\is{E (\sigma) f}{f} = \sum_{u \in V}
\chi_\sigma(\|\slam e_u\|) |f(u)|^2, \quad \sigma \in
\borel{\rbb_+}, f \in \ell^2(V).
   \end{align*}
Since $(|\slam| f) (u) = \|\slam e_u\| f(u)$ for all
$u \in V$ and $f \in \dz{\slam}$ (cf.\
\cite[Proposition 3.4.3]{j-j-s}), we deduce that
   \begin{align}   \label{mod1}
\is{E(\sigma)|\slam|f}{|\slam|f} = \sum_{u \in V}
\chi_\sigma(\|\slam e_u\|) \|\slam e_u\|^2 |f(u)|^2
   \end{align}
for all $\sigma \in \borel{\rbb_+}$ and $f \in
\dz{\slam}$. In view of the equality $V^\circ=
\bigsqcup_{u\in V} \dzi u$ (cf.\ \cite[Proposition
2.1.2]{j-j-s}), we have
   \begin{align}     \notag
\is{E(\sigma)\slam f}{\slam f} & = \sum_{u \in
V^\circ} \chi_\sigma(\|\slam e_u\|) |\lambda_u|^2
|f(\pa{u})|^2
   \\  \label{mod2}
& = \sum_{u \in V} \Big(\sum_{v \in \dzi{u}}
\chi_\sigma(\|\slam e_v\|) |\lambda_v|^2\Big) |f(u)|^2
   \end{align}
for all $\sigma \in \borel{\rbb_+}$ and $f \in
\dz{\slam}$.

(i) Since, by \cite[Proposition 3.1.3(v)]{j-j-s},
$\escr \subseteq \dz{\slam}$, we infer from
\eqref{mod1} and \eqref{mod2} that the inequality
\eqref{inc1} holds with $A=\slam$ if and only if
   \begin{align} \label{wlnier}
\chi_\sigma(\|\slam e_u\|) \|\slam e_u\|^2 \Le c
\sum_{v \in \dzi{u}} \chi_\sigma(\|\slam e_v\|)
|\lambda_v|^2, \quad u \in V,\, \sigma \in
\borel{\rbb_+}.
   \end{align}
First we show that \eqref{wlnier} implies
\eqref{abc2}. Suppose \eqref{wlnier} holds. Fix $u \in
V$ and define the set $\varOmega_u=\{\|\slam e_v\|
\colon v \in \dzi{u}\} \subseteq \rbb_+$. We may
assume that $\|\slam e_u\| \neq 0$. Then, by
\eqref{lamtauf}, $\dzi{u} \neq \varnothing$. If
$\varOmega_u = \rbb_+$, then clearly $\dzie{u} \neq
\varnothing$. If $\varOmega_u \neq \rbb_+$, then
substituting $\sigma=\{t\}$ with $t \in \rbb_+
\setminus \varOmega_u$ into \eqref{wlnier}, we deduce
that $\|\slam e_u\| \neq t$. Hence $\rbb_+\setminus
\varOmega_u \subseteq \rbb_+ \setminus \{\|\slam
e_u\|\}$, which yields $\dzie{u} \neq \varnothing$.
This proves \eqref{abc}. By substituting
$\sigma=\{\|\slam e_u\|\}$ into \eqref{wlnier}, we
obtain \eqref{abc2}.

Now we show that \eqref{abc2} implies \eqref{wlnier}.
Fix $u \in V$ and $\sigma \in \borel{\rbb_+}$. Without
loss of generality we can assume that $\|\slam e_u\|
\neq 0$ and $\|\slam e_u\| \in \sigma$. Then, by
\eqref{abc2},
   \begin{align*}
\chi_\sigma(\|\slam e_u\|) \|\slam e_u\|^2 \Le c
\sum_{v\in \dzie{u}} \chi_\sigma(\|\slam e_v\|)
|\lambda_v|^2 \hspace{1ex}\Le c \sum_{v \in \dzi{u}}
\chi_\sigma(\|\slam e_v\|) |\lambda_v|^2,
   \end{align*}
which shows that \eqref{wlnier} holds. This completes
the proof of (i).

(ii) We can argue as in the proof of (i). The details
are left to the reader.
   \end{proof}
The following characterization of quasinormality of
weighted shifts on directed trees generalizes that of
\cite[Proposition 8.1.7]{j-j-s} to the case of
unbounded operators. The present proof is quite
different from that for bounded operators.
   \begin{cor}\label{qn}
Let $\slam$ be a densely defined weighted shift on a
directed tree $\tcal$ with weights $\lambdab =
\{\lambda_v\}_{v \in V^\circ}$. Then the following two
conditions are equivalent\/{\em :}
   \begin{enumerate}
   \item[(i)] $\slam$ is quasinormal,
   \item[(ii)] $\|\slam
e_u\|=\|\slam e_v\|$ for all $u \in V$ and $v \in \dzi
u$ such that $\lambda_v \neq 0$.
   \end{enumerate}
Moreover, if $V^\circ \neq \varnothing$ and $\lambda_v
\neq 0$ for all $v\in V^\circ$, then $\slam$ is
quasinormal if and only if $\|\slam\|^{-1}\slam$ is an
isometry.
   \end{cor}
   \begin{proof}
(i)$\Rightarrow$(ii) By Theorem \ref{chq2}, $\slam$ is
weakly quasinormal with $\cfr_{\slam}\Le 1$. One can
deduce from \cite[Proposition 3.1.3]{j-j-s} and
Theorem \ref{wsquasineq}(i), applied to $c=1$, that
$\lambda_v = 0$ for all $v\in \dzi{u} \setminus
\dzie{u}$ and $u \in V$. This implies (ii).

(ii)$\Rightarrow$(i) By our present assumption
$\lambda_v = 0$ for all $v\in \dzi{u} \setminus
\dzie{u}$ and $u \in V$. This implies that
\eqref{abc2} holds with $c=1$. Hence, by Theorem
\ref{wsquasineq}(i), $\slam$ is weakly quasinormal
with $\cfr_{\slam} \Le 1$. Applying Theorems
\ref{chq3}(iii) and \ref{quasiwkwkw2} yields (i).

Arguing as in the proof of \cite[Proposition
8.1.7]{j-j-s}, we deduce the ``moreover'' part of the
conclusion from the equivalence
(i)$\Leftrightarrow$(ii).
   \end{proof}
We will show by example that there are unbounded
quasinormal weighted shifts on directed trees (cf.\
Example \ref{eunb}).

   The following corollary is closely related to
Proposition \ref{qinvb}.
   \begin{cor} \label{sleu}
If $\slam$ is a quasinormal weighted shift on a
directed tree $\tcal$, then $\slam(\escr) \subseteq
\dz{|\slam|^{\alpha}}$ for every positive real number
$\alpha$.
   \end{cor}
   \begin{proof}
By Corollary \ref{qn} and \cite[Proposition
3.1.3]{j-j-s}, for every $u \in V$,
   \begin{multline} \label{olin}
\sum_{v\in V} \|\slam e_v\|^{2\alpha} |(\slam
e_u)(v)|^2 \overset{\eqref{lamtauf}}= \sum_{v\in
V^\circ} \|\slam e_v\|^{2\alpha} |\lambda_v|^2
e_u(\paa(v))
   \\
= \sum_{v\in \dzi{u}} \|\slam e_v\|^{2\alpha}
|\lambda_v|^2 = \|\slam e_u\|^{2\alpha} \sum_{v\in
\dzi{u}} |\lambda_v|^2 = \|\slam e_u\|^{2(\alpha+1)}.
   \end{multline}
Since, by \cite[Proposition 3.4.3]{j-j-s}, $\escr
\subseteq \dz{|\slam|^\alpha}$ and $|\slam|^\alpha e_u
= \|\slam e_u\|^\alpha e_u$ for every $u \in V$, we
deduce that a function $f \in \ell^2(V)$ belongs to
$\dz{|\slam|^\alpha}$ if and only if $\sum_{v \in V}
\|\slam e_v\|^{2\alpha} |f(v)|^2 < \infty$ (consult
the proof of \cite[Lemma 2.2.1]{j-j-s}). This combined
with \eqref{olin} gives $\slam e_u \in
\dz{|\slam|^\alpha}$ for $u\in V$, which completes the
proof.
   \end{proof}
   \section{\label{Exam}Examples}
This section provides examples of weighted shifts on
directed trees that illustrate the subject of this
paper. We begin by considering the case of quasinormal
operators. It follows from Corollary \ref{qn} that
quasinormal weighted shifts on directed trees with
nonzero weights are automatically bounded. However, if
some of the weights are allowed to be zero, then
quasinormal weighted shifts may be unbounded. Below,
we construct an example of an injective quasinormal
weighted shift on a directed binary tree whose
restriction to $\ell^2(\des{u})$ is unbounded for
every $u\in V$.
   \begin{exa} \label{eunb}
Let $\tcal$ be a directed tree with root such that for
every $u \in V$, the set $\dzi u$ has exactly two
vertices. By \cite[Corollary 2.1.5]{j-j-s}, we have
   \begin{align*}
V^\circ = \bigsqcup_{n=1}^\infty \dzin{n}{\koo}.
   \end{align*}
We define recursively a sequence
$\{\{\lambda_v\}_{v\in \dzin{n}{\koo}}\}_{n=1}^\infty$
of systems of nonnegative real numbers. We begin with
$n=1$. If $v_1,v_2 \in \dzi{\koo}$ and $v_1 \neq v_2$,
then we set $\lambda_{v_1}=0$ and $\lambda_{v_2}=1$.
Suppose that we have constructed the systems
$\{\lambda_v\}_{v\in \dzin{j}{\koo}} \subseteq
[0,\infty)$ for $j=1, \ldots, n$. To construct
$\{\lambda_v\}_{v\in \dzin{n+1}{\koo}}$, note that
(cf.\ \cite[(6.1.3)]{j-j-s})
   \begin{align} \label{num1}
\dzin{n+1}{\koo} = \bigsqcup_{u \in \dzin{n}{\koo}}
\dzi{u}.
   \end{align}
Fix $u \in \dzin{n}{\koo}$. By our assumption $\dzi{u}
= \{v,w\}$ with $v\neq w$. If $\lambda_u=0$, then we
set $\lambda_v=0$ and $\lambda_w = n+1$. If
$\lambda_u\neq0$, then we set $\lambda_v=0$ and
$\lambda_w = \lambda_u$. In view of \eqref{num1}, the
recursive procedure gives us the system
$\lambdab:=\{\lambda_v\}_{v \in V^\circ}$. Let $\slam$
be the weighted shift on $\tcal$ with weights
$\lambdab$. By \cite[Proposition 3.1.3]{j-j-s},
$\slam$ is densely defined. It is a routine matter to
verify that $\slam$ satisfies the condition (ii) of
Corollary \ref{qn}. Hence the operator $\slam$ is
quasinormal. It is easily seen, by using
\cite[Proposition 3.1.8]{j-j-s}, that for every $u \in
V$, the operator $\slam|_{\tiny{\mbox{\sc
Lin}}\{e_v\colon v \in \des{u}\}}$ (which acts in
$\ell^2(\des{u})$) is unbounded. The injectivity of
$\slam$ follows from \cite[Proposition 3.1.7]{j-j-s}.
   \end{exa}
Now we show how to construct non-quasinormal weighted
shifts on certain directed trees that are weakly
quasinormal as well as non-weakly quasinormal weighted
shifts on the same directed trees that satisfy the
condition (i) of Theorem \ref{abcon} (with $A=\slam$).
Note that this is not possible for classical weighted
shifts (that is, weighted shifts on the directed tress
\mbox{$(\zbb_+, \{(n,n+1)\colon n \in \zbb_+\})$} and
\mbox{$(\zbb, \{(n,n+1)\colon n \in \zbb\})$}, cf.\
\cite[Remark 3.1.4]{j-j-s}), because by Theorem
\ref{wsquasineq} and Corollary \ref{qn} every
classical weighted shift which satisfies the condition
(i) of Theorem \ref{abcon} is automatically
quasinormal. Hence, non-quasinormal classical weighted
shifts (many such exist) do not satisfy this
condition. We also show that for every $c \in
(1,\infty)$, there exists an injective weighted shift
$\slam$ on a directed tree such that $\cfr_{\slam}=c$,
were $\cfr_{\slam}$ is understood as in Section
\ref{bbcc}. In Examples \ref{nq-ineq} and
\ref{nq-ineq2}, we consider the cases of bounded and
unbounded non-hyponormal operators with the properties
mentioned above. All this can also be achieved in the
class of hyponormal operators, as is shown in Example
\ref{nq-ineq3}.
   \begin{center}
   \includegraphics[width=12cm]
   {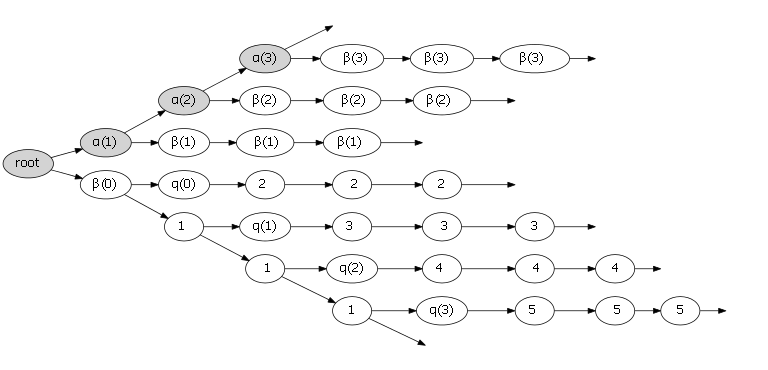}

   {\small {\sf Figure 1}}
   \end{center}
   \begin{exa} \label{nq-ineq}
Let $\tcal$ be the directed tree as in Figure 1, where
$\{\alpha(n)\}_{n=1}^\infty$,
$\{\beta(n)\}_{n=0}^\infty$, $\{q(n)\}_{n=0}^\infty$
and $\{\gamma(n)\}_{n=0}^\infty$ are sequences of
positive real numbers satisfying the following two
conditions
   \begin{gather}   \label{sumcia}
\alpha(n)^2 + \beta(n-1)^2=1, \quad n \in \nbb,
   \\   \label{sumcia2}
1+q(n)^2=\gamma(n)^2, \quad n\in\zbb_+.
   \end{gather}
Let $\slam$ be the weighted shift on $\tcal$ with
weights given by Figure 1. By \cite[Proposition
3.1.3]{j-j-s}, $\slam$ is densely defined (and closed
as a weighted shift on a directed tree). It follows
from \eqref{sumcia}, \eqref{sumcia2} and
\cite[Proposition 3.1.8]{j-j-s} that $\slam$ is
bounded if and only if the sequence
$\{q(n)\}_{n=0}^\infty$ is bounded.

We first note that $\slam$ is not hyponormal. Indeed,
since
   \begin{align*}
\sum_{v \in \dzi{u_i}} \frac{|\lambda_{v}|^2}{\|\slam
e_v\|^2} = 1 + \alpha(i+1)^2 > 1, \quad i \in \nbb,
   \end{align*}
where $u_i$ is the vertex corresponding to
$\alpha(i)$, we infer from \cite[Theorem 5.1.2 and
Remark 5.1.5]{j-j-s} that $\slam$ is not hyponormal.

Suppose now that $\inf_{n\in \nbb} \alpha(n)= 0$. Then
$\slam$ is not weakly quasinormal. Indeed, otherwise
by Theorem \ref{wsquasineq}\,(i) applied to $u=u_i$
with $i\in\zbb_+$ ($u_0:=\koo$), there exists $c>0$
such that $1 \Le c \alpha(i+1)^2$ for all $i \in
\zbb_+$, which is impossible. It follows from
\eqref{sumcia} and \eqref{sumcia2} that \eqref{abc3}
holds, which in view of Theorem \ref{wsquasineq}\,(ii)
implies that
   \begin{align*}
\is{E(\cdot)|\slam|f}{|\slam|f} \ll \is{E(\cdot)\slam
f}{\slam f}, \quad f \in \dz{\slam}.
   \end{align*}

Fix $c \in (1,\infty)$. Suppose now that $\inf_{n\in
\nbb} \alpha(n) = 1/\sqrt{c}$ and $q(i)^{-2} + 1 \Le
c$ for all $i \in \zbb_+$ (we still assume that
\eqref{sumcia} and \eqref{sumcia2} are satisfied). It
is easily seen that \eqref{abc2} holds. Hence, by
Theorem \ref{wsquasineq}\,(i), $\slam$ is weakly
quasinormal with $\cfr_{\slam} \Le c$. We show that
$\cfr_{\slam}=c$. Indeed, by Theorem
\ref{wsquasineq}\,(i), $1 \Le \cfr_{\slam}
\alpha(i)^2$ for all $i \in \nbb$, which implies that
$1/\sqrt{\cfr_{\slam}} \Le 1/\sqrt{c}$, and thus
$\cfr_{\slam} \Ge c$.

Finally, note that the so-constructed operator $\slam$
can be made bounded or unbounded according to our
needs, still maintaining its properties discussed
above. This can be achieved by considering bounded or
unbounded sequences $\{q(n)\}_{n=0}^\infty$.
   \end{exa}
   \begin{center}
   \includegraphics[width=8cm]
   {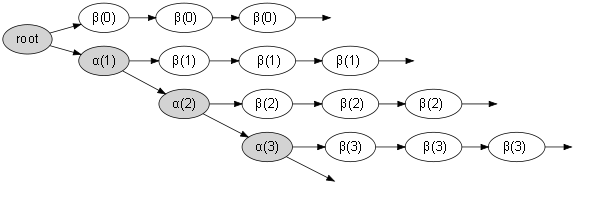} 

   {\small {\sf Figure 2}}
   \end{center}
   \vspace{1ex}
   \begin{exa} \label{nq-ineq2}
Let $\{\alpha(n)\}_{n=1}^\infty$ and
$\{\beta(n)\}_{n=0}^\infty$ be sequences of positive
real numbers that satisfy \eqref{sumcia}. The reader
can easily convince himself that the weighted shifts
$\slam$ on the directed tree $\tcal$ given by Figure
2, which are less complicated than that in Figure 1,
have all the properties specified in Example
\ref{nq-ineq} (each of which depends on the choice of
weights) except for unboundedness, namely $\slam$ are
always bounded.
   \end{exa}
   \vspace{1ex}
   \begin{center}
   \includegraphics[width=8cm]
   {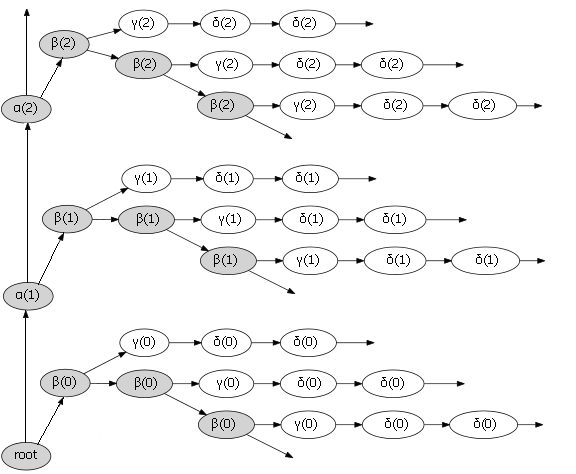} 

   \vspace{1ex} {\small {\sf Figure 3}}
   \end{center}
   \vspace{0.5ex}
   \begin{exa}  \label{nq-ineq3}
The previous two constructions can be modified so as
to obtain examples of weighted shifts $\slam$ on a
directed tree with nonzero weights which have all the
properties stated in Example \ref{nq-ineq} (each of
which depends on the choice of weights) except for
non-hyponormality, namely $\slam$ are hyponormal.
Since the main idea remains the same, we skip the
discussion of the present example. For the reader's
convenience we draw a figure that contains the
necessary data (cf.\ Figure 3). The sequences
$\{\alpha(n)\}_{n=1}^\infty$,
$\{\beta(n)\}_{n=0}^\infty$,
$\{\gamma(n)\}_{n=0}^\infty$ and
$\{\delta(n)\}_{n=0}^\infty$ consist of positive real
numbers that satisfy \eqref{sumcia} and the following
three conditions
   \begin{align}   \notag
&\delta(n)^2 = \beta(n)^2 + \gamma(n)^2, \quad n\in
\zbb_+,
   \\ \label{abgd1}
&\frac{\beta(n)^2}{\delta(n)^2} + \alpha(n+1)^2 < 1,
\quad n\in \zbb_+.
   \\ \label{abgd2}
&\delta(n)>1, \quad n \in \zbb_+.
   \end{align}
It is worth pointing out that under the assumption
\eqref{sumcia}, the conditions \eqref{abgd1} and
\eqref{abgd2} are equivalent.
   \end{exa}
   \section{Remarks and further results}
The absolute continuity approach developed in this
paper in the context of quasinormal operators can be
generalized to the case of other classes of operators.
The class of $q$-quasinormal operators, a particular
case of $q$-deformed operators introduced by \^Ota in
\cite{ota-q} (see also
\cite{ota-q-hyp,ota-fhsz,ota-fhsz-q-sub}) in
connection with the theory of quantum groups (see
\cite{kl-sch}), is well-suited for our purposes.

Let $q$ be a positive real number. Following
\cite{ota-q}, we say that a closed densely defined
operator $A$ in a complex Hilbert space $\hh$ is {\em
$q$-quasinormal} if $U|A| \subseteq \sqrt{q} \, |A|U$,
where $A = U|A|$ is the polar decomposition of $A$ (or
equivalently if and only if $U|A| = \sqrt{q} \, |A|U$;
cf.\ \cite[Lemma 2.2]{ota-q}). In view of
\cite[Theorem 2.5]{ota-q}, a closed densely defined
operator $A$ in $\hh$ is $q$-quasinormal if and only
if
   \begin{align*}
UE(\sigma)=E(\psi_q^{-1}(\sigma))U, \quad \sigma \in
\borel{\rbb_+},
   \end{align*}
where $E$ is the spectral measure of $|A|$ and
$\psi_q\colon \rbb_+\to \rbb_+$ is a Borel function
given by $\psi_q(x)=\sqrt{q} \, x$. The above suggests
the following generalization.
   \begin{pro}\label{intert}
Let $A$ be a closed densely defined operator in $\hh$,
$A=U|A|$ be its polar decomposition and $E$ be the
spectral measure of $|A|$. Suppose $\phi$ and $\psi$
are Borel functions from $\rbb_+$ to $\rbb_+$. Then
the following conditions are
equivalent\,\footnote{\;$E(\varphi^{-1}(\cdot))$
stands for the spectral measure $\borel{\rbb_+} \ni
\sigma \longmapsto E(\varphi^{-1}(\sigma))\in
\ogr{\hh}$.}{\em :}
   \begin{enumerate}
   \item[(i)]
 $UE(\varphi^{-1}(\cdot)) = E(\psi^{-1}(\cdot))U$,
   \item[(ii)] $U\varphi(|A|) \subseteq \psi(|A|) U$,
   \item[(iii)] $E(\psi^{-1}(\cdot))A \subseteq
AE(\varphi^{-1}(\cdot))$.
   \end{enumerate}
   \end{pro}
   \begin{proof}
(i)$\Leftrightarrow$(ii) Use the measure transport
theorem (cf.\ \cite[Theorem 5.4.10]{b-s}) and the
``intertwining'' version of \cite[Theorem 6.3.2]{b-s}.

(i)$\Leftrightarrow$(iii) Adapt the proof of
\cite[Proposition 1]{StSz1}.
   \end{proof}
Below we assume that $\phi,\psi\colon \rbb_+ \to
\rbb_+$ are fixed Borel functions. Arguing exactly as
in the proofs of Theorems \ref{chq2}, \ref{chq3},
\ref{abcon} and \ref{quasiwkwkw2}, and using
Proposition \ref{intert} together with its proof, we
obtain the following more general results. It is also
worth pointing out that the ``moreover'' parts of
Theorems \ref{chq3} and \ref{abcon} can be easily
adapted to this new context as well. We leave the
details to the reader.
   \begin{thm}
Let $A$ be a closed densely defined operator in $\hh$
and $E$ be the spectral measure of $|A|$. Then the
following conditions are equivalent\/{\em :}
   \begin{enumerate}
   \item[(i)] $E(\psi^{-1}(\cdot))A \subseteq
AE(\varphi^{-1}(\cdot))$,
   \item[(ii)]
$\is{E(\psi^{-1}(\cdot))Af}{Af} =
\is{E(\varphi^{-1}(\cdot))|A|f}{|A|f}$ for all $f \in
\dz{A}$,
   \item[(iii)]
$\is{E(\psi^{-1}(\cdot))Af}{Af} \ll
\is{E(\varphi^{-1}(\cdot))|A|f}{|A|f}$ for all $f \in
\dz{A}$.
   \end{enumerate}
   \end{thm}
   \begin{thm}
Let $A$ be a closed densely defined operator in $\hh$,
$E$ be the spectral measure of $|A|$ and $c\in
\rbb_+$. Then the following conditions are
equivalent\/{\em :}
   \begin{enumerate}
   \item[(i)] $\is{E(\varphi^{-1}(\cdot))|A|f}{|A|f} \Le
c \is{E(\psi^{-1}(\cdot))Af}{Af}$ for all $f \in
\dz{A}$,
   \item[(ii)] there exists $T \in \ogr{\hh}$ such that
   \begin{align*}
\text{$TA=|A|$, $\|T\| \Le \sqrt{c}$ and
$TE(\psi^{-1}(\cdot)) = E(\varphi^{-1}(\cdot))T$,}
   \end{align*}
   \item[(iii)] there exists $T \in \ogr{\hh}$ such that
   \begin{align*}
\text{$TA=|A|$, $\|T\| \Le \sqrt{c}$ and $T\psi(|A|)
\subseteq \varphi(|A|)T$.}
   \end{align*}
   \end{enumerate}
   \end{thm}
   \begin{thm}
Let $A$ be a closed densely defined operator in $\hh$
and $E$ be the spectral measure of $|A|$. Then the
following conditions are equivalent\/{\em :}
   \begin{enumerate}
   \item[(i)] $\is{E(\varphi^{-1}(\cdot))|A|f}{|A|f} \ll
\is{E(\psi^{-1}(\cdot))Af}{Af}$ for all $f \in
\dz{A}$,
   \item[(ii)] there exists a $($unique$)$ linear
map $T_0 \colon \hscr_0 \to \ob{|A|}$ such that $T_0
A=|A|$ and $T_0E(\psi^{-1}(\cdot))|_{\hscr_0} =
E(\varphi^{-1}(\cdot))T_0$, where
   \begin{align*}
\hscr_0 = \lin\big\{E(\psi^{-1}(\sigma))Af \colon
\sigma \in \borel{\rbb_+ }, \, f \in \dz{A}\big\}.
   \end{align*}
   \end{enumerate}
   \end{thm}
   \begin{thm}
Let $A$ be a closed densely defined operator in $\hh$
and $E$ be the spectral measure of $|A|$. Then the
following conditions are equivalent{\em :}
   \begin{enumerate}
   \item[(i)] $E(\psi^{-1}(\cdot))A \subseteq
AE(\varphi^{-1}(\cdot))$,
   \item[(ii)] $\is{E(\varphi^{-1}(\cdot))|A|f}{|A|f} \Le
\is{E(\psi^{-1}(\cdot))Af}{Af}$ for all $f \in
\dz{A}$.
   \end{enumerate}
   \end{thm}
Substituting $\varphi=$ the identity function on
$\rbb_+$ and $\psi=\psi_q$ into the above theorems, we
obtain characterizations of $q$-quasinormal operators,
``$q$-variants'' of weakly quasinormal operators and
operators satisfying the ``$q$-version'' of the
condition (i) of Theorem \ref{abcon}.
   \subsection*{Acknowledgements} A substantial part of
this paper was written while the first and the third
authors visited Kyungpook National University during
the spring of 2012. They wish to thank the faculty and
the administration of this unit for their warm
hospitality.
   \bibliographystyle{amsalpha}
   
   \end{document}